\documentclass[10pt]{amsart}
\usepackage{graphicx}
\usepackage{amsmath,amsfonts,amssymb,amsthm}
\usepackage{multicol}

\newtheorem{theorem}{Theorem}[section]
\newtheorem{conjecture}[theorem]{Conjecture}
\newtheorem{proposition}[theorem]{Proposition}
\newtheorem{lemma}[theorem]{Lemma}

\newtheorem{definition}[theorem]{Definition}

\numberwithin{equation}{section} \numberwithin{theorem}{section}

\newcommand{\ind}{{\large \bf{1}}}
\newcommand{\T}{\ensuremath{\mathcal{T}}}
\newcommand{\Z}{\mathbb Z}
\newcommand{\R}{\mathbb R}
\newcommand{\G}{\mathcal G}
\newcommand{\X}{\mathcal X}

\newcommand{\x}{\mathbf x}

\newcommand{\y}{\mathbf y}

\newcommand{\0}{\mathbf 0}

\newcommand{\M}{\mathcal M}
\newcommand{\I}{\mathcal I}

\title{Limiting shapes for a non-abelian sandpile growth model and related cellular automata}

\author{Anne Fey $^{*,\dagger}$
\and
Haiyan Liu $^*$}

\begin{document}

\bigskip

\begin{abstract}
We present limiting shape results for a non-abelian variant of the abelian sandpile growth model (ASGM), some of which have no analog in the ASGM. One of our limiting shapes is an octagon. In our model, mass spreads from the origin by the toppling rule in Zhang's sandpile model. Previously, several limiting shape results have been obtained for the ASGM using abelianness and monotonicity as main tools. As both properties fail for our model, we use a new proof technique: in our main proof, we introduce several cellular automata to mimic our growth model.
\end{abstract}

\maketitle


\bigskip

{\small \emph{$^*$ VU University Amsterdam, De Boelelaan 1081a, 1081 HV Amsterdam, The Netherlands,} \texttt{hliu@few.vu.nl,ac.den.boer@few.vu.nl}}

{\small \emph{$^\dagger$CWI, P.O. Box 94079, 1090 GB Amsterdam, The Netherlands,} \texttt{a.c.fey@cwi.nl}}



\bigskip

\section{Introduction}

We consider the following setup. Start with a pile of mass $n \geq 1$ at the origin of the rectangular grid $\Z^d$, and mass $h<1$ at every other site, 
where $n$ and $h$ are real numbers. Mass may be moved by `splitting piles', that is, one may take all the mass from one site, and divide it 
evenly among its neighbors. One can only split `unstable' piles of mass at least 1, and one can not stop before all the piles have mass less than 1. We call $\T$ the - possibly random - set of sites where at least one split was performed.

Will the mass spread over all of $\Z^d$, or will $\T$ be a finite subset of $\Z^d$? In the first case, how does it spread? In the last case, what is the size and shape of this set, depending on $h$ and $n$? How do these answers depend on the order of splitting?

This splitting game is related to abelian sandpile growth models. In fact, the `splitting pile' rule is the same as the toppling rule
in Zhang's sandpile model \cite{zhang}, a sandpile model that has not been considered as a growth model before. A notorious difficulty of moving 
mass by the Zhang toppling rule is the non-abelianness, that is, the end result depends on the order of topplings \cite{FMQR,haiyan}. 

In other studies of sandpile growth models \cite{FLP,shapes,lionel}, ample use was made of the freedom, by abelianness, to choose some convenient 
toppling order. In this way, information could be derived about limiting shapes, growth rates and about whether an explosion occurs or not. The term `explosion' is introduced in \cite{FLP}: if an explosion occurs, then the mass from the origin spreads over all of $\Z^d$, and every site topples infinitely often.
For `robust backgrounds', that is, values of $h$ such that an explosion never occurs, one can examine the growth rate and existence of a limiting 
shape as $n \to \infty$. Roughly speaking, the growth rate is the radius of $\T$ as a function of $n$, and if the set $\T$, properly 
scaled, converges to a deterministic shape as $n \to \infty$, then that is the limiting shape. 

In \cite{lionel}, the main topic is the rotor router model. In the rotor router model, the mass consists of discrete grains, so that only discrete 
values of $h$ are possible. All sites except the origin start empty, and at all sites there is a router which points to a neighbor. Instead of splitting, a pile that has at least one grain may give a grain to the neighbor indicated by the router. The router then rotates to point to the next neighbor, for a cyclic ordering of all the neighbors. It was proved that the limiting shape is a sphere, and the growth rate 
is proportional to $n^{1/d}$. The proof makes use of properties of Green's function for simple random walk.
In this paper, we will demonstrate that the method of \cite{lionel} can be adapted for our model, resulting in a growth rate proportional to $n^{1/d}$ for all $h < 0$; see Theorem \ref{ball}.

In the abelian sandpile model (for background, see \cite{dhar,meester}), the mass also consists of discrete sand grains. Instead of splitting, a 
pile that consists of at least $2d$ grains may {\em topple}, that is, give one grain to each of its neighbors. Thus, $2d$ consecutive rotor router moves of one site equal one abelian sandpile toppling of that site. In \cite{shapes}, making use of this equivalence, it was shown that for the abelian sandpile growth model with $h \to -\infty$, the limiting shape is a sphere. 
Moreover, it was proved that for $h = 2d-2$, the limiting shape is a cube. In \cite{FLP}, it was proved that for all $h \leq 2d-2$, the growth rate 
is proportional to $n^{1/d}$.

All these proofs heavily rely on the abelianness, which is an almost routine technique for the abelian sandpile model, but fails to hold in our 
case. For instance, consider the case $d=1$, $n=4$, and $h = 0$. If we choose the parallel updating rule that is common in cellular automata, 
in each time step splitting every unstable site, then we end up with
$$
\ldots,0,1/2,3/4,3/4,0,3/4,3/4,1/2,0,\ldots
$$
However, if we for example choose to split in each time step only the leftmost unstable site, then we end up with
$$
\ldots, 0,1/2,1/2,7/8,3/4,0,3/4,5/8,0,\ldots
$$
For arbitrary splitting order, it may even depend solely on the order if there occurs an explosion or not; see the examples in \cite{haiyan}, 
Section 4.1.
In this paper, we focus primarily on the parallel splitting order, but several of our results are valid for arbitrary splitting order.

Another complicating property of our model is that unlike the abelian sandpile model, we have no monotonicity in $h$ nor in $n$.
In the abelian sandpile growth model, it is almost trivially true that for fixed $h$, $\T$ is nondecreasing in $n$, and for fixed $n$, $\T$ is nondecreasing in $h$. For our growth model however, this is false, even if we fix the parallel splitting order. 
Consider the following examples for $d=1$: For the first example, fix $h = 23/64$. 
Then if $n = 165/32 \approx 5.16\ldots$, we find that $\T$ is the interval $[-5,-4,\ldots,4,5]$. However, if $n = 167/32 \approx 5.22\ldots$, then $\T = [-4,-3,\ldots,3,4]$.
For the second example, fix $n = 343/64$. 
Then if $h = 21/64 \approx 0.33\ldots$, we find that $\T = [-5,-4,\ldots,4,5]$, but if $h = 23/64 \approx 0.36\ldots$, then $\T = [-4,-3,\ldots,3,4]$.



One should take care when performing numerical simulations for this model. Since in each splitting, a real number is divided 
by $2d$, one quickly encounters rounding errors due to limited machine accuracy. All simulations presented in this paper were done by performing 
exact calculations in binary format. 

This article is organized as follows: after giving definitions in Section \ref{definitionssection}, we present our results and some short proofs in Section \ref{mainresultssection}. Our main result is that for $h$ explosive and with the parallel splitting order, the splitting model exhibits several different limiting shapes as $t \to \infty$. We find a square, a diamond and an octagon. These results are stated in our main Theorem \ref{differentshapes}, which is proved in Section \ref{explosionssection}. Sections \ref{hvaluessection} and \ref{ballsection} contain the remaining proofs of our other results. Finally, in Section \ref{openproblems} we comment on some open problems for this model. 

\section{Definitions}
\label{definitionssection}

In this section, we formally define the splitting model. While we focus primarily on the parallel order of splitting, some of our results are also valid for more general splitting order. Therefore, we give a general definition of the
splitting model, with the splitting automaton (parallel splitting order) as a special case. 

For $n\in[0, \infty)$ and $h\in(-\infty, 1)$, $\eta_n^{h}$ is the configuration given by 

$$
\eta_n^{h}(\x) = \left\{ \begin{array}{ll}
									n & \mbox{ if } \x = \0,\\
									h & \mbox{ if } \x \in \Z^d\setminus \0.\\
								\end{array} \right.
$$ 
For every $t = 0,1,2,\ldots$, and for fixed $n$ and $h$, 
$\eta_{t}$ is the configuration at time $t$, and the initial configuration is $\eta_{0}=\eta_n^h$. We interpret $\eta_{t}(\x)$ as the mass at site $\x$ at time $t$. 

We now describe how $\eta_{t+1}$ is obtained from $\eta_{t}$, for every $t$.
Denote by $\mathcal{U}_{t} = \{\x: \eta_{t}(\x) \geq 1\}$ the set of all {\em unstable} sites at time $t$.
$\mathcal{S}_{t+1}$ is a (possibly random) subset of $\mathcal{U}_{t}$. We say that $\mathcal{S}_{t+1}$ is the set of sites that {\em split} at time $t+1$. Then the configuration at time $t+1$ is for all $\x$ defined by 
$$
\eta_{t+1}(\x) = \eta_{t}(\x) \left(1 - \ind_{\x \in \mathcal{S}_{t+1}}\right) + \frac 1{2d} \sum_{\y \in \mathcal{S}_{t+1}} \eta_{t}(\y) \ind_{\y \sim \x}.
$$

The {\em splitting order} of the model determines how we choose $\mathcal{S}_{t+1}$, given $\mathcal{U}_{t}$, for every $t$. For example, if we have the parallel splitting order then we choose $\mathcal{S}_{t+1} = \mathcal{U}_{t}$ for every $t$. In this case, we call our model the {\em splitting automaton}. Some of our results are also valid for other splitting orders. In this paper we only consider splitting orders with the following properties: At every time $t$, $\mathcal{S}_{t+1}$ is non-empty unless $\mathcal{U}_{t}$ is empty, and
for every $\x$ that is unstable at time $t$, there exists a finite time $t_0$ such that $x\in\mathcal{S}_{t+t_0}$. 
For example, we allow the random splitting order where $\mathcal{S}_{t+1}$ contains a single element of $\mathcal{U}_{t}$, chosen uniformly at random from all elements of $\mathcal{U}_{t}$. With this splitting order, at each time step a single site splits, randomly chosen from all unstable sites. This splitting order is valid because $\mathcal{U}_t$ increases slowly enough. Since only the neighbors of sites that split can become unstable, we have for any splitting order that $\mathcal{U}_{t+1} \subseteq \mathcal{U}_{t} + \partial \mathcal{U}_t$, where with $\partial \X$ for a set $\mathcal{X} \subset \Z^d$, we denote the set of sites that are not in $\X$, but have at least one neighbor in $\X$. But when every time step only a single site splits, at most $2d$ sites can become unstable, so that $|\mathcal{U}_{t+1}| \leq |\mathcal{U}_{t}| + 2d$.

We are interested in the properties of 
$$
\T_{t} = \bigcup_{0 < t' \leq t}\mathcal{S}_{t'},
$$
all the sites that split at least once until time $t$, as well as 
$$
\T=\lim_{t \to \infty}\T_{t}.
$$ 

For a fixed splitting order, we say that $\eta_n^h$ {\em stabilizes} if in the limit $t \to \infty$, for every $\x$, the total number of times that site $\x$ splits is finite. Note that if $\eta_n^h$ does not stabilize, then every site splits infinitely often. We also remark that in order to show that $\eta^h_n$ does not stabilize, it suffices to show that $\T = \Z^d$. Namely, if $\T = \Z^d$, then every site splits infinitely often.
Otherwise, there is a site $\x$ and a time $t$ such that $\x$ does not split at any time $t'>t$, but each of its neighbors $\y_i$, $i = 1,2, \ldots, 2d$, splits at some time $t_i \geq t$. But then at time $\max_i t_i$, $\x$ is not stable, because it received at least mass $\frac 1{2d}$ from each of its neighbors. Therefore, $\x$ must split again.

We call $\eta_n^h$ {\em stabilizable} if $\eta_n^h$ stabilizes almost surely (The ``almost surely" refers to randomness in the splitting order). 
As defined in \cite{FLP}, 
\begin{definition}
The background $h$ is said to be {\em robust} if $\eta_n^h$ is stabilizable for all finite $n$; 
it is said to be {\em explosive} if there is a $N^h<\infty$ such that for all $n\geq N^h$, $\eta_n^h$ is not stabilizable.
\end{definition}

{\bf Remark } 
We expect the splitting model for every $h$ to be either robust or explosive, independent of the splitting order (see Section \ref{criticalh}). However, even for a fixed splitting order we cannot a priori exclude intermediate cases where the background is neither robust nor explosive. 
For example, since the splitting model is not monotone in $n$, it might occur for some $h$ that for every $n$, 
there exist $n_1 > n_2 > n$ such that $\eta_{n_1}^h$ is stabilizable, but $\eta_{n_2}^h$ is not.
  
Finally, we give our definition of a limiting shape. In \cite{FLP,shapes, lionel}, limiting shape results were obtained for $\T$ in the limit $n \to \infty$, for robust background. In this paper however, we present limiting shape results for $\T_{t}$ with $n$ fixed, parallel splitting order and explosive background. We study the limiting behavior in $t$ rather than in $n$. Accordingly, we have a different definition for the limiting shape. 

 
Let $\mathbf{C}$ denote the cube of radius 1/2 centered at the origin $\{\x \in \R^d: \max_i x_i \leq 1/2\}$. Then $\x + \mathbf{C}$ is the same cube centered at $\x$, and by $\mathcal{V} + \mathbf{C}$ we denote the volume $\bigcup _{\x \in \mathcal{V}} (\x + \mathbf{C})$.

 
\begin{definition}
Let $\mathcal{V}_t$, with $t = 0,1,\ldots$ be a sequence of sets in $\Z^d$. Let $\mathbf{S}$ be a deterministic shape in $\R^d$, scaled such that $\max_{x_i} \{ \x \in \mathbf{S}\} = 1$.
Let $\mathbf{S}_\epsilon$ and $\mathbf{S}^\epsilon$ denote respectively the inner and outer $\epsilon$-neighborhoods of $\mathbf{S}$.
We say that $\mathbf{S}$ is the {\em limiting shape} of $\mathcal{V}_t$, 
if there is a scaling function $f(t)$, and for all $\epsilon>0$ there is a $t^\epsilon$ such that for all $t>t^\epsilon$, 
$$
 \mathbf{S}_\epsilon \subseteq f(t) \left(\mathcal{V}_{t} + \mathbf{C} \right) \subseteq \mathbf{S}^\epsilon.
$$
If $\mathbf{S}$ is the limiting shape of $\T_t$, then we say that $\mathbf{S}$ is the limiting shape of the splitting automaton.
\end{definition}

\medskip

\section{Main results}
\label{mainresultssection}

\begin{figure}
\centering
\includegraphics[width=.315\textwidth]{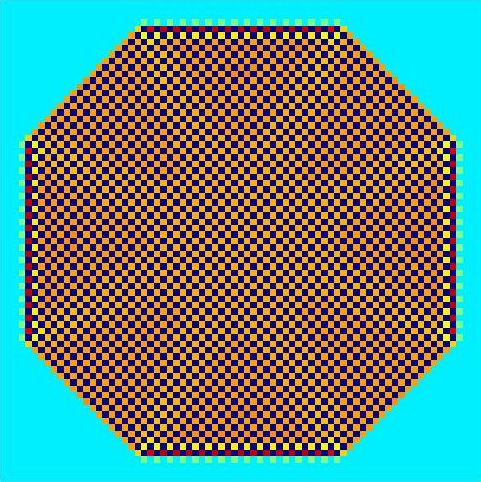}
\hspace{0.1cm}
\includegraphics[width=.315\textwidth]{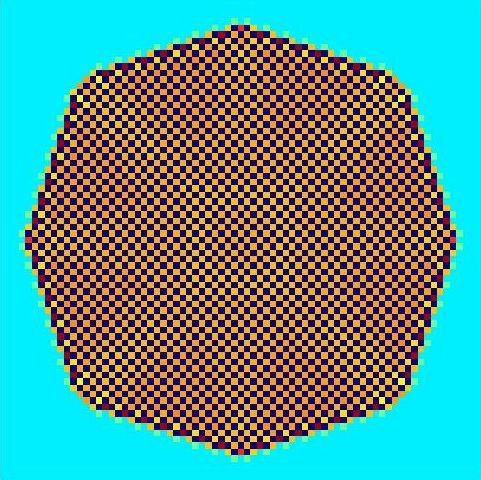}
\hspace{0.1cm}
\includegraphics[width=.315\textwidth]{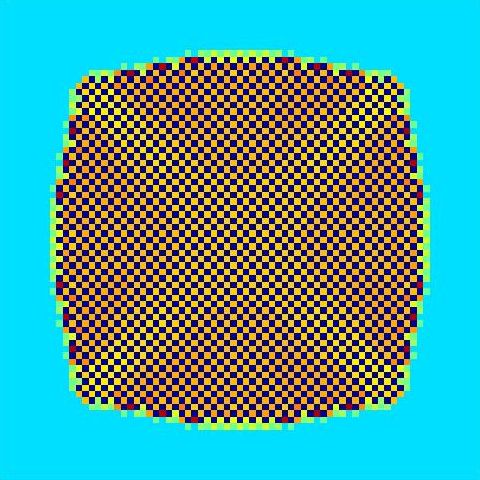}

\smallskip

\includegraphics[width=.315\textwidth]{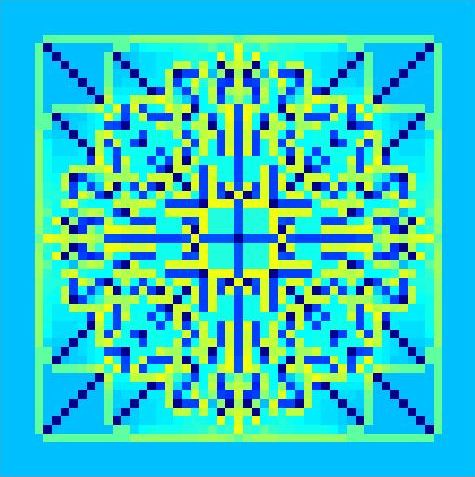}
\hspace{0.1cm}
\includegraphics[width=.315\textwidth]{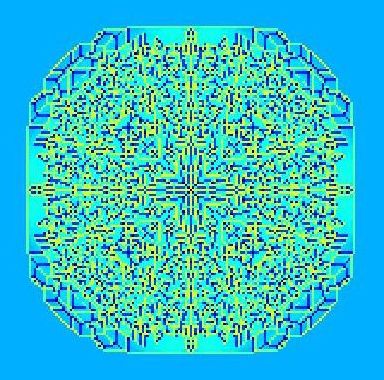}
\hspace{0.1cm}
\includegraphics[width=.315\textwidth]{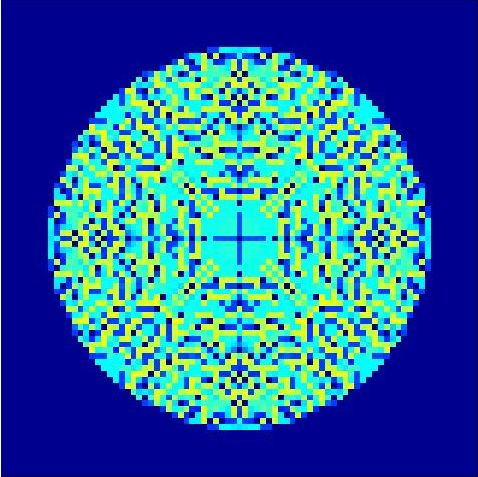}
\caption{The splitting automaton for different values of $h$. ``Warmer" color is larger mass; dark yellow, orange or red cells have mass $\geq 1$. Top row: $h = 47/64 \approx 0.734$; $h = 1495/2048 \approx 0.73$ and $h = 727/1024 \approx 0.71$, each with $n=8$ and after 50 time steps. Bottom row: $h = 1/2$ and $n = 256$; $h = 511/1024 \approx 0.499$ and $n = 2048$; $h = 0$ and $n = 2048$, each after the model stabilized.}
\label{plaatjes}
\end{figure}

We have observed - see Figure \ref{plaatjes} - that varying $h$ has a striking effect on the dynamics of the splitting automaton. For large values of $h$, $\T_{t}$ appears to grow in time with linear speed, resembling a polygon, but which polygon depends on the value of $h$. Figures \ref{plaatjes} and \ref{explosions} support the conjecture that for the parallel splitting order, as $t \to \infty$, there are many possible different limiting shapes depending on $h$. Our main result, Theorem \ref{differentshapes}, is that there are at least three different polygonal limiting shapes, for three different intervals of $h$. 
For small values of $h$, the splitting model stabilizes; see Theorem \ref{hvalues}. In between, there is a third regime that we were not able to characterize. It appears that for $h$ in this regime, $\T_t$ keeps increasing in time, but does not have a polygonal limiting shape. We comment on this in Section \ref{openproblems}.

\begin{figure}
\begin{center}
\includegraphics[width=.8\textwidth]{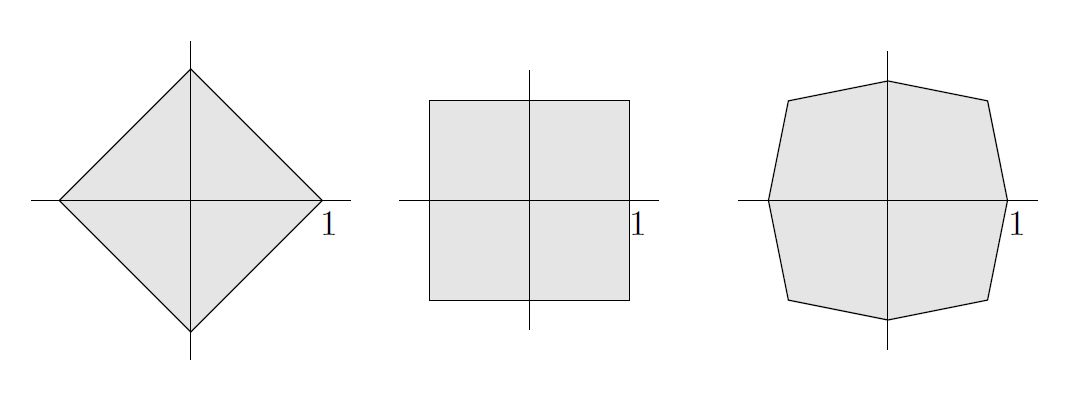}
\end{center}
\caption{The diamond $\mathbf{D}$, the square $\mathbf{Q}$, and the octagon $\mathbf{O}$.}
\label{threeshapes}
\end{figure}

\medskip
\clearpage

Let $\mathbf{D}$ be the diamond in $\R^d$ with radius 1 centered at the origin; let $\mathbf{Q}$ be the square with radius 1 in $\R^2$ centered at the origin, and let $\mathbf{O}$ be the octagon in $\R^2$ with vertices (0,1) and ($\frac 56$,$\frac 56$), and the other six vertices follow from symmetry. See Figure \ref{threeshapes} for these shapes.

\begin{theorem}~
\begin{enumerate}
\item  The limiting shape of the splitting automaton on $\Z^d$, for $1-\frac 1{2d} \leq h < 1$ and $n \geq 1$, is the diamond $\mathbf{D}$.
The scaling function is $f(t) = \frac 1t$.

\item The limiting shape of splitting automaton on $\Z^2$, for $ 7/10 \leq h < 40/57$ and $4-4h \leq n \leq 16-20h$, is the square $\mathbf{Q}$. The scaling function is $f(t) = \frac 2t$.

\item The limiting shape of the splitting automaton on $\Z^2$, for $ 5/7 \leq h < 13/18$ and $n=3$, is the octagon $\mathbf{O}$.
The scaling function is $f(t) = \frac 5{3t}$.

\end{enumerate}
\label{differentshapes}
\end{theorem}

In the abelian sandpile growth model, $h=2d-1$ is the only possible explosive background value, and our proof of part 1 also works for that 
situation. However, the second two parts have no analog in the abelian sandpile growth model.

Our proof uses the method of mimicking the splitting automaton with a finite state space cellular automaton; we consider our three explicit results 
as introductory examples for this method. We expect that with this method, many more limiting shape results can be obtained.

Next, we characterize several regimes of $h$ for the splitting model on $\Z^d$.

\medskip

\begin{theorem}
\label{hvalues}
In the splitting model on $\mathbb{Z}^d$,
\begin{enumerate}
\item The background is explosive if $h \geq 1-\frac 1{2d}$,
\item The background is robust if $h < \frac 12$,
\item In the splitting automaton, for $d \geq 2$, there exist constants $C_d < 1 - \frac{3}{4d+2}$ such that the background is explosive if $h \geq C_d$.
\end{enumerate}
\end{theorem}

We give the proof of the first part here, because it is a very short argument. The proof of parts 2 and 3 will be given in Section \ref{hvaluessection}, where we give the precise form of $C_d$. We do not believe that this bound is sharp. From the simulations for $d=2$ for example, a transition between an explosive and robust regime appears to take place at $h = 2/3$, while in our proof of part 3, $C_2 = 13/19 \approx 0.684$.

\medskip

\noindent{\it Proof of Theorem \ref{hvalues}, part 1}

If $n \geq 1$, then the origin splits, so then $\T$ is not empty. Now suppose that $\T$ is a finite set. Then there exist sites outside $\T$ that have a neighbor in $\T$. Such a site received at least $\frac 1{2d}$, but did not split. For $h \geq 1-\frac 1{2d}$, this is a contradiction.
\qed

\medskip

Note that Theorem \ref{hvalues} does not exclude the possibility that there exists, for $d$ fixed, a single critical value of $h$ that separates 
explosive and robust backgrounds, independent of the splitting order. We only know this in the case $d=1$, for which the first two bounds are equal.

We give another result that can be proved by a short argument:

\medskip

\begin{theorem}
\label{rectangular}
In every splitting model on $\mathbb{Z}^d$, for every $n \geq 1$ and $h \geq 1-\frac 1d$, if the model stabilizes then $\T$ is a $d$-dimensional rectangle.
\end{theorem}

\begin{proof}
Because $n \geq 1$, $\T$ is not empty. Suppose that $\T$ is not a rectangle. Then, as is not hard to see, there must exist a site that did not split, but has at least two neighbors that split. Therefore, its final mass is at least $h + \frac 1d$, but strictly less than 1. This can only be 
true for $h < 1-\frac 1d$.
\end{proof}

In the case of the parallel splitting order, we additionally have symmetry. Thus, for $h \geq 1- \frac 1d$, $\T$ is a cube. 
We remark that the above proof also works for the abelian sandpile model, thus considerably simplifying the proof of Theorem 4.1 (first 2 parts) in \cite{shapes}.

\medskip

Our final result gives information on the size and shape of $\T$ when $h < 0$. This theorem is similar to Theorem 4.1 of \cite{lionel}.
Let $\mathbf{B}_r$ denote the Euclidean ball in $\R^d$ with radius $r$, and let $\omega_d$ be the volume of $\mathbf{B}_1$. 

\medskip
	
\begin{theorem}~

\begin{enumerate}
\item ({\it Inner bound}) For all $h<1$,
$$
\mathbf{B}_{c_1r-c_2}\subset \T,
$$
with $c_1=(1-h)^{-1/d}$,
$r=(\frac{n}{\omega_d})^{1/d}$ and $c_2$ a constant which depends only on $d$;

\item ({\it Outer bound}) When $h < 0$, for every $\epsilon>0$,
$$
\T \subset \mathbf{B}_{c_1'r+c_2'},
$$
with $c_1'=(\frac{1}{2}-\epsilon-h)^{-1/d}$, $r=(\frac{n}{\omega_d})^{1/d}$ and $c_2'$ a constant which depends only on $\epsilon$, $h$ and $d$.
\end{enumerate}
\label{ball}
\end{theorem}

\section{Limiting shapes in the explosive regime}
\label{explosionssection}

\begin{figure}
\centering
\includegraphics[width=.315\textwidth]{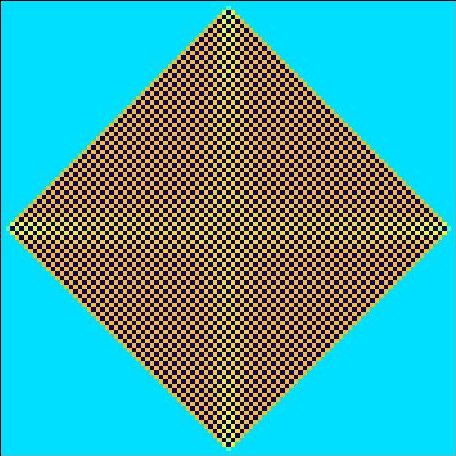} \hspace{0.7mm}
\includegraphics[width=.315\textwidth]{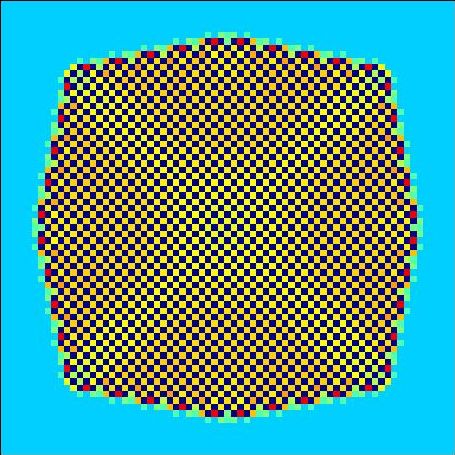} \hspace{0.7mm}
\includegraphics[width=.315\textwidth]{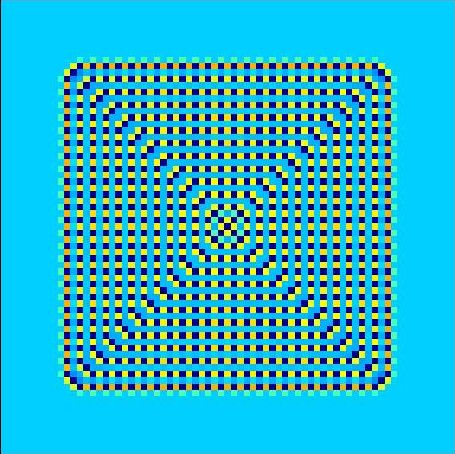}
\caption{The splitting automaton after 50 time steps, with $n=3$ and $h=\frac 34 = 0.75$ (diamond), $h= \frac{23}{32} \approx 0.72$ (octagon), and $h = \frac{359}{512} \approx 0.701$ (square).}
\label{explosions}
\end{figure}

In this section, we will prove Theorem \ref{differentshapes}. Each part of this theorem is stated for $h$ in a certain interval, and the first two parts for $n$ in a certain interval. That means the theorem is stated for uncountably many possible initial configurations. However, we will show that we do not need to know all the exact masses to determine $\T_{t}$ for a certain $t$. For each part of the theorem, we will introduce a labeling of sites, using only a finite number of labels. It will suffice to know the labels of all sites at time $t$, to determine $\T_{t'}$ for all $t' \geq t$.
 
We will see in each case that the time evolution in terms of the labels is a lot more enlightening than in terms of the full information contained in $\eta_t$. In each case, we can identify a certain recurrent pattern of the labels. Our limiting shape proofs will be by induction in $t$, making use of these recurrent patterns. 

The label of a site at time $t$ will depend on its own label at time $t-1$ plus those of its neighbors at time $t-1$. We will specify the labels at $t=0$, and the transition rules for the labels. In other words, for each part of the theorem we will introduce a finite state space cellular automaton, that describes the splitting automaton for certain intervals of $h$ and $n$ in terms of the labels.

A cellular automaton is defined by giving its state space $\mathcal{S}$, its initial configuration $\xi_0$, and its transition rules. By $\xi_t(\x)$, we denote the label of site $\x$ at time $t$. The state space will consist of a finite number of labels. The transition rules specify how the label of each cell changes as a function of its own current label and those of its neighbors. A cellular automaton evolves from the initial configuration in discrete time; each time step, all cells are updated in parallel according to the transition rules. 

We use the following notation for the transition rules. Let $s'$ and $s,s_1,s_2,\ldots$ denote labels in $\mathcal{S}$ (not necessarily all different). By
$$
s \oplus  s_1,s_2,\ldots,s_{2d} \longrightarrow s',
$$
we mean that if a cell has label $s$, and there is a permutation of the labels of its $2d$ neighbors equal to $\{s_1,s_2,\ldots,s_{2d}\}$, then the label of this cell changes to $s'$. By $*$, we will denote an arbitrary label. For example, if we have a transition rule
$$
s \oplus  s_1,s_1,*,\ldots,* \longrightarrow s',
$$
then the label of a cell with label $s$ will change into $s'$ if  at least two of its neighbors have label $s_1$, irrespective of the labels of the other neighbors.

We first give the proof for part 1, which is the simplest case. In fact, in this case the splitting model is equivalent to (1,d) bootstrap percolation: $\T_{t+1}$ is the union of $\T_{t}$ with all sites that have at least one neighbor in $\T_{t}$. The proof we give below will seem somewhat elaborate for such a simple case. That is because we use this case to illustrate our method of labels and cellular automata.

We will need the following observation, which can be proved by induction in $t$:

\begin{lemma}
We call $\x$ an odd site if $\sum_i x_i$ is odd, otherwise we call $\x$ an even site.
Then in the splitting automaton, even sites only split at even times, and odd sites only at odd times.
\label{checkerboardlemma}
\end{lemma}

\noindent{\it Proof of Theorem \ref{differentshapes}, part 1}

We begin by defining the diamond cellular automaton. 

\begin{definition}~

The {\em diamond cellular automaton} has state space $\{e,\hbar,u\}$. We additionally use the symbol $s$ to denote $e$ or $\hbar$.
In the initial configuration, every cell has label $\hbar$, only the origin has label $u$.
The transition rules are:

\begin{enumerate}
\item
$\hbar \oplus  s,\ldots,s \longrightarrow \hbar$
\item
$s \oplus   u,\ldots,u \longrightarrow u$
\item
$u \oplus   s,\ldots,s \longrightarrow e$
\item
$\hbar \oplus   u,*,\ldots,* \longrightarrow u$
\end{enumerate}

\end{definition}

The above set of transition rules is sufficient to define the diamond cellular automaton, because, as we will demonstrate below, other combinations of 
cell and neighborhood labels do not occur.

Let $\G_t$, the growth cluster of the cellular automaton, be the set of all cells that do not have label $\hbar$ at time $t$. We will first prove the limiting shape result for $\G_t$, and then demonstrate that if $1-\frac 1{2d} \leq h <1$, then $\G_t$ is the same set as $\T_{t} \cup \partial \T_{t}$ for every $t$. 

Let $\mathcal{D}_r$ be the diamond $\{\x \in \mathbb{Z}^d: \sum_i |x_i| \leq r\}$.
To prove the limiting shape result, we will show by induction in $t$ that $\G_t = \mathcal{D}_t$, so that the limiting shape of $\G_t$ is $\mathbf{D}$, with scaling function $f(t) = \frac 1t$.

Our induction hypothesis is that at time $t$, $\xi_t$ is as follows (see figure \ref{diamondindhyp}): all sites $x \in \mathcal{D}_{t}$ have label $u$
if $(\sum_i x_i -t) \mod 2 = 0$, and label $e$ otherwise. All other sites have label $\hbar$. If this claim is true for all $t$, then we have 
$\G_{t} = \mathcal{D}_t$.

\begin{figure}[h]
\begin{center}
\begin{tabular}{c|c|c|c|c}
 & & u &  &  \\
\hline
 & u & e & u &  \\
\hline
u & e & u & e & u \\
\hline
 & u & e & u &  \\
\hline
 & & u &  &  \\
\end{tabular}
\end{center}
\caption{The induction hypothesis for the diamond cellular automaton at $t=2$. Labels not shown are $\hbar$.}
\label{diamondindhyp}
\end{figure}

As a starting point, we take $t=0$. At that time, the origin has label $u$, and all other cells have label $\hbar$. Therefore, the hypothesis is true at $t=0$.

Now suppose the hypothesis is true at time $t$. Then all sites with label $e$ have $2d$ neighbors with label $u$, 
therefore by the second transition rule they will have label $u$ at time $t+1$. All sites with label $u$ have $2d$ neighbors with label $e$ or $\hbar$, 
therefore by the third rule they will have label $e$ at time $t+1$. All sites in $\mathcal{D}_{t+1} \setminus \mathcal{D}_{t}$ have label $\hbar$ and 
at least one neighbor with label $u$. Therefore, by the fourth rule they will have label $u$ at $t+1$. Other labels do not change, by the first rule. 
This concludes the induction, and moreover shows that our set of transition rules is sufficient to define the diamond cellular automaton.

Finally, we show that if $1-\frac 1{2d} \leq h <1$, then $\G_t$ is the same set as $\T_{t} \cup \partial \T_{t}$. To compare the configurations $\eta_t$ and $\xi_t$, we give a mapping 
$$
\mathcal{M}_d:\{e,\hbar,u,s\} \to \I,
$$ 
where $\I$ is the set of intervals $\{[a,b): a \leq b, ~ a,b \in [0,\infty]\}$,
that maps the state space of the diamond cellular automaton to the mass values of the splitting automaton:

\begin{eqnarray*}
\mathcal{M}_d(e) & = & 0\\
\mathcal{M}_d(\hbar) & = & h\\
\mathcal{M}_d(s) & = & [0,1)\\
\mathcal{M}_d(u) & = & [1,\infty)\\
\end{eqnarray*}

For a fixed diamond cellular automaton configuration $\xi$, define 
$$
\mathcal{M}_d^\xi =\{\eta: \eta(\x)\in\mathcal{M}_d(\xi(\x)), \textrm{ for all } \x \in \mathbb{Z}^d\}.
$$

With this mapping, the initial configuration of the splitting automaton $\eta^h_n$ is in $\mathcal{M}_d^{\xi_0}$, for all $n \geq1$ and $1-\frac 1{2d} \leq h <1$. We will show by induction in $t$ that $\eta_t$ is in $\mathcal{M}_d^{\xi_t}$ for all $t$. 
Suppose that at time $t$, $\eta_t$ is in $\mathcal{M}_d^{\xi_t}$. 

We check one by one the transition rules:
\begin{itemize}
\item{(rule 1)} If in the splitting automaton a site has mass $h$, and none of its neighbors split, then its mass does not change. 
This is true for all $h<1$.
\item{(rule 2)} If in the splitting automaton a stable site has $2d$ unstable neighbors, then its receives at least $2d$ times $\frac 1{2d}$, 
therefore its mass will become at least 1. This is true for all $h$.
\item{(rule 3)} If in the splitting automaton a site has mass at least 1, then it splits. If none of its neighbors splits, then its mass will become 0. 
This is true for all $h$.
\item{(rule 4)} If in the splitting automaton a site with mass $h$ has at least one neighbor that splits, it receives at least $\frac 1{2d}$. Therefore, it will become unstable only if $1-\frac 1{2d} \leq h <1$.
\end{itemize}

Therefore, if $\eta_t$ is in $\M_d^{\xi_t}$, $1-\frac 1{2d} \leq h <1$ and $n \geq 1$, then $\eta_{t+1}$ is in $\M_d^{\xi_{t+1}}$. This completes the induction.

Finally, by the following observations:
\begin{itemize}
\item only the label $u$ maps to mass 1 or larger, so a site is in $\T_{t}$ if and only if it has had label $u$ at least once before $t$,
\item the label of a site changes into another label if and only if at least one neighbor has label $u$,
\item if a site does not have label $\hbar$ at time $t$, then it cannot get label $\hbar$ at any time $t' \geq t$, 
\end{itemize}
we can conclude that $\G_t$ of the diamond cellular automaton is the same set as $\T_{t} \cup \partial \T_{t}$ for the splitting automaton with $1-\frac 1{2d} \leq h <1$ and $n \geq 1$.
\qed

\medskip

We now give the proofs of the remaining two parts; note that in these next two proofs, we are in dimension 2. We will need more elaborate cellular automata, in which there are several labels for unstable sites. For example, it is important to know whether the mass of a splitting site is below or above $4(1-h)$: if its neighbor has mass $h$ then in the first case it might not become unstable, but in the second case, it will. 

\medskip

\noindent{\it Proof of Theorem \ref{differentshapes}, part 2}

We begin by defining the square cellular automata. 

\begin{definition}~

The {\em square cellular automaton} has state space $\{e,\hbar,p,m,m',c,d\}$. We additionally use the symbol $s$ to denote a label that is $e$, $\hbar$ or $p$.
In the initial configuration, every cell has label $\hbar$, only the origin has label $d$. 
The transition rules are:

\begin{multicols}{2}
\begin{enumerate}
\item
$\hbar \oplus   s,s,s,s \longrightarrow \hbar$
\item
$p \oplus   s,s,s,s \longrightarrow p$
\item
$c \oplus   *,*,*,* \longrightarrow c$
\item
$m \oplus   *,*,*,* \longrightarrow e$
\item
$d \oplus   *,*,*,* \longrightarrow c$
\item
$m' \oplus  *,*,*,* \longrightarrow c$
\item
$\hbar \oplus   d,s,s,s \longrightarrow m$
\item
$\hbar \oplus   m,s,s,s \longrightarrow p$
\item
$p \oplus   m,m,m',s \longrightarrow d$
\item
$\hbar \oplus   d,m,s,s \longrightarrow d$
\item
$\hbar \oplus   m,m,s,s \longrightarrow d$
\item
$e \oplus   d,d,c,p \longrightarrow m'$
\end{enumerate}
\end{multicols}

\end{definition}

The above set of transition rules is sufficient to define the cellular automaton, because, as we will demonstrate below, other combinations of cell 
and neighborhood labels do not occur.

Recall that the growth cluster $\G_t$ is the set of all cells that do not have label $\hbar$ at time $t$.
To prove the limiting shape result for the growth cluster of the square cellular automaton, we use induction. Let $\mathcal{C}_r \in \Z^2$ be the square $\{(i,j):|i| \leq r, |j| \leq r\}$.
Let $\zeta_r$ be the following configuration (see Figure \ref{squareindhyp}): 
\begin{itemize}
\item All sites in $\mathcal{C}_{r-1}$ have label $c$. 
\item The labels in $\mathcal{C}_{r}\setminus \mathcal{C}_{r-1}$ are $d$, if $(i-j) \mod 2 = 0$, and $e$ otherwise.
\item The labels outside $\mathcal{C}_{r}$ are $p$ if they have a neighbor with label $e$, and $\hbar$ otherwise.
\end{itemize}

\begin{figure}[h]
\begin{center}
\begin{tabular}{c|c|c|c|c|c|c}
 & & p &  & p & & \\
\hline
 & d & e & d & e & d & \\
\hline
p & e & c & c & c & e & p \\
\hline
 & d & c & c & c & d &  \\
\hline
p & e & c & c & c & e & p \\
\hline
 & d & e & d & e & d & \\
 \hline
  & & p &  & p & & \\
\end{tabular}
\end{center}
\caption{The configuration $\zeta_r$, used in the induction hypothesis for the square cellular automaton, for $r=2$. Labels not shown are $\hbar$.}
\label{squareindhyp}
\end{figure}

Our induction hypothesis is that for every even $t$, $\xi_t = \zeta_{t/2}$.
The initial configuration $\xi_0$ of the square cellular automaton is $\zeta_0$.
Now suppose that at some even time $t$, $\xi_t = \zeta_{t/2}$. Then by using the transition rules, one can check that at time $\xi_{t+2}$ will 
be $\zeta_{t/2+1} = \zeta_{(t+2)/2}$. This completes the induction.

The shape result now follows: for every $t$, $\mathcal{C}_{t/2} \subseteq \G_t \subseteq \mathcal{C}_{t/2 +1}$, so that the limiting shape of $\G_t$ is the square $\mathbf{Q}$, with scaling function $f(t) = \frac 2t$.

\medskip

Finally, to show that $\G_t$ for the square cellular automaton is equal to $\T_{t} \cup \partial \T_{t}$ for the splitting automaton with $4-4h \leq n \leq 16-20h$ and $ 7/10 \leq h < 40/57$, we give a mapping 
$$
\mathcal{M}_s:\{e,\hbar,p,m,m',c,d,s\} \to \I,
$$ 
that maps the state space of the square cellular automaton to the mass values of the splitting automaton:

\begin{eqnarray*}
\mathcal{M}_s(e) & = &0\\
\mathcal{M}_s(\hbar) & = & h\\
\mathcal{M}_s(s) & = & [0,1)\\
\mathcal{M}_s(p) & = & [h+\frac 14,1)\\
\mathcal{M}_s(m) & = & [1,4-4h)\\
\mathcal{M}_s(m') & = & [0,12-15h)\\
\mathcal{M}_s(c) & = & [0,16-20h)\\
\mathcal{M}_s(d) & = & [4-4h,16-20h)\\
\end{eqnarray*}
For all $h<3/4$, these intervals are nonempty, moreover, $4-4h > 1$. 

With this mapping, one may check that $\eta^h_n$ is in $\mathcal{M}_s^{\xi_0}$. By induction in $t$, we will show that $\eta_t$ is in $\mathcal{M}_s^{\xi_t}$ for all $t$.  
Suppose at time $t$, $\eta_{t}$ is in $\M_s^{\xi_{t}}$.

We check one by one the transition rules:

\begin{itemize}
\item{(rule 1)} If in the splitting automaton a site has mass $h$, and none of its neighbors split, then its mass does not change. This is true for 
all $h<1$.
\item({rules 2-5)} If an unstable site splits, then by Lemma \ref{checkerboardlemma}, its neighbors do not split. Therefore, it will become empty. 
If a cell has split at least once, then from that time on it cannot receive sand from more than 4 neighbors before splitting itself. Therefore, 
no cell that split at least once can gain mass greater than $16-20h$. This is true for all $h<1$.
\item{(rule 6)} $h + \frac 14 [4-4h,16-20h) \rightarrow [1,4-4h)$. This is true for all $h<1$.
\item{(rule 7)} $h + \frac 14 [1,4-4h) \rightarrow [h+1/4,1)$. This is true for all $h<1$.

\item{(rule 8)} $[h+\frac 14, 1) + \frac 12 [1,4-4h) + \frac 14[0,12-15h) \rightarrow [h+\frac 34,6-\frac{23h}4)$.
We have that $[h+\frac 34,6-\frac{23h}4) \subseteq [4-4h,16-20h)$ if $13/20 \leq h \leq 40/57$.
\item{(rule 9)} $h + \frac 12 [1,4-4h) \rightarrow [h+\frac 12,2-h)$.
We have that $[h+\frac 12,2-h) \subseteq [4-4h,16-20h)$ if $7/10 \leq h \leq 14/19$.
\item{(rule 10)} $h + \frac 14 [4-4h,16-20h) + \frac 14[1,4-4h) \rightarrow [\frac 54,5-5h)$.
We have that $[\frac 54,5-5h) \subseteq [4-4h,16-20h)$ if $11/16 \leq h \leq 11/15$.
\item{(rule 11)} $ \frac 12 [4-4h,16-20h) + \frac 14 [0,16-20h) \rightarrow [2-2h,12-15h)$.
We have that $[2-2h,12-15h) \subseteq [0,12-15h)$ if $h < 1$.
\end{itemize}

Therefore, if $\eta_{t}$ is in $\M_s^{\xi_t}$, $4-4h \leq n \leq 16-20h$ and $7/10 \leq h \leq 40/57$, then $\eta_{t+1}$ is in $\M_s^{\xi_{t+1}}$. This completes the induction.

Finally, by the following observations:
\begin{itemize}
\item the labels $m$ and $d$ map to an interval in $[1,\infty)$, so a site is in $\T_{t}$ if it has had label $m$ or $d$ at least once before $t$,
\item the label of a site with label $\hbar$ changes into another label if and only if at least one neighbor has label $m$ or $d$,
\item if a site does not have label $\hbar$ at time $t$, then it cannot get label $\hbar$ at any time $t' \geq t$, 
\end{itemize}
we can conclude that if $4-4h \leq n \leq 16-20h$ and $7/10 \leq h \leq 40/57$, then $\G_t$ for the square cellular automaton is the same set as $\T_{t} \cup \partial \T_{t}$.
\qed

\medskip 

For the final part, we first perform 8 time steps in the splitting automaton before we describe its further evolution as a finite state space cellular automaton. Otherwise, we would need many more labels and transition rules.

\medskip

\noindent{\it Proof of Theorem \ref{differentshapes}, part 3}

We begin by defining the octagon cellular automaton.

\begin{definition}~

The {\em octagon cellular automaton} has state space $\{e,\hbar,p,m,d,d',d!,c,c',q,q'\}$. We additionally use the symbol $s$ to denote a label 
that is $e$, $\hbar$ or $p$, and the symbol $u$ to denote a label that is any of the other.

The initial configuration is given in the table below. We show only the first quadrant (left bottom cell is the origin). The rest follows by symmetry. All labels not shown are $\hbar$.

\medskip
\begin{center}
\begin{tabular}{|c|c|c|c|c|c|c|c}
p & &  &  &  &  &  &\\
\hline
e & d! & p & m &  &  & & \\
\hline
$c'$ & e & c & e & $d'$ &  & & \\
\hline
e & c & e & c & e & m & & \\
\hline
c & e & c & e & c & p &  &\\
\hline
e & c & e & c & e & d! &  &\\
\hline
c & e & c & e & $c'$ & e &p  &\\
\hline
\end{tabular}
\end{center}

\medskip

The transition rules are:

\begin{multicols}{2}
\begin{enumerate}
\item
$\hbar \oplus  s,s,s,s \longrightarrow \hbar$
\item
$u \oplus  *,*,*,* \longrightarrow e$
\item
$\hbar \oplus   m,s,s,s \longrightarrow p$
\item
$\hbar \oplus   d,s,s,s \longrightarrow m$
\item
$\hbar \oplus   d',s,s,s \longrightarrow m$
\item
$\hbar \oplus   d!,s,s,s \longrightarrow m$
\item
$\hbar \oplus   q,s,s,s \longrightarrow d$
\item
$\hbar \oplus   q,m,s,s \longrightarrow d!$
\item
$\hbar \oplus   q',m,s,s \longrightarrow d!$
\item
$\hbar \oplus   m,d',s,s \longrightarrow d'$
\item
$\hbar \oplus   d',d',s,s \longrightarrow d'$
\item
$p \oplus   d!,m,c,s \longrightarrow q'$
\item
$p \oplus   m,m,c,s \longrightarrow q$
\item
$p \oplus   d,m,c,s \longrightarrow q$
\item
$e \oplus   q',c,d',s \longrightarrow c$
\item
$e \oplus   d!,d!,c',s \longrightarrow c$
\item
$e \oplus   q,c,d!,s \longrightarrow c$
\item
$e \oplus   q',c,d!,s \longrightarrow c$
\item
$e \oplus   u,u,u,u \longrightarrow c$ \\
\mbox except in the following case:
\item
$e \oplus   m,q,q,c \longrightarrow c'$
\end{enumerate}
\end{multicols}
\end{definition}

Our proof that the growth cluster of the octagon cellular automaton has $\mathbf{O}$ as limiting shape, is by induction. We will show that there is a pattern that repeats every 10 time steps. To describe this pattern, we introduce two subconfigurations that we call `tile' and `cornerstone'. They are given in Figure \ref{tiles}. We say the subconfiguration is at position $(x,y)$, if the left bottom cell has coordinates $(x,y)$.


\begin{figure}[h]
\begin{multicols}{2}
\begin{center}

\begin{tabular}{|c|c|c|c|c|}
\hline
  $p$ & $\hbar$ & $\hbar$ & $\hbar$ & $\hbar$ \\
   \hline
 $e$ & $d!$ & $p$ & $m$ & $\hbar$ \\
   \hline
 $c$ & $e$ & $c$ & $e$ & $q'$ \\
 \hline
  \end{tabular}

\begin{tabular}{|c|c|c|c|}
\hline
  $p$ & $\hbar$ & $\hbar$ & $\hbar$  \\
 \hline
 $e$ & $d'$ & $\hbar$ & $\hbar$ \\
 \hline 
 $c$ & $e$ & $d'$ & $\hbar$ \\
 \hline
 $e$ & $c$ & $e$ & $p$ \\
  \hline
\end{tabular}

\end{center}
\end{multicols}
  
\caption{A `tile' (left), and a `cornerstone' (right).}
\label{tiles}
\end{figure}

To specify a configuration, we will only specify the cells $\x = (x,y)$ with $y \geq x \geq 0$; the rest follows by symmetry.

\begin{definition}
We define the configuration $\chi^i$, with $i \in \{0,1,\ldots \}$, as follows:
\begin{itemize}
\item there is a cornerstone at position $(5(i+1),5(i+1))$,
\item for every $j = 0,\ldots,i$, there is a tile at position $(5(i-j),7+5i+j)$,
\item the leftmost tile is different, namely, $\chi^i(0,7+6i) = c'$,
\item for every cell $(x,y)$ such that it is not in a tile or cornerstone, but every directed path from $(x,y)$ to $(0,0)$ intersects a tile or cornerstone, $\chi^i(x,y) = \hbar$,
\item for every other cell $(x,y)$, $\chi^i(x,y) = e$ if $(x+y)\mod 2 = 0$, and $c$ otherwise.
\end{itemize}
\end{definition}

\begin{figure}[h]
\centering
\includegraphics[width=.31\textwidth]{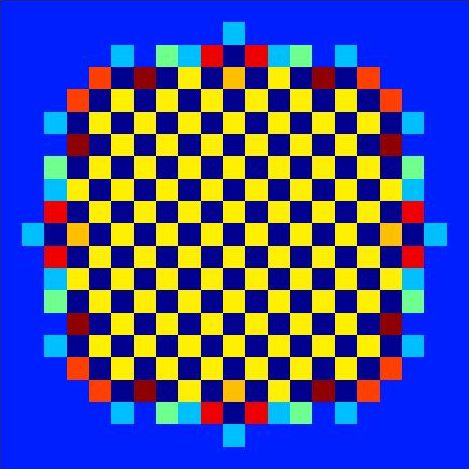} \hspace{0.5mm}
\includegraphics[width=.31\textwidth]{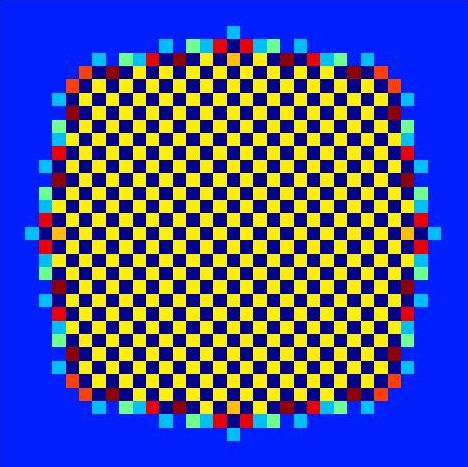} \hspace{0.5mm}
\includegraphics[width=.31\textwidth]{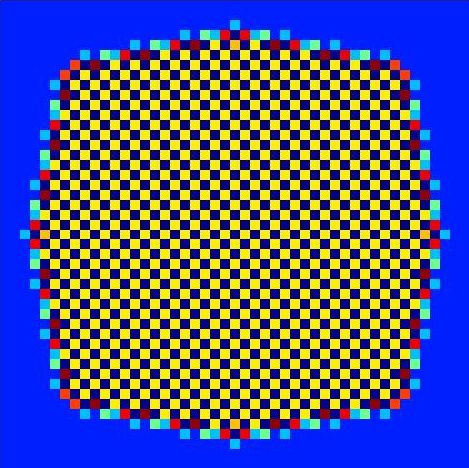} 

\smallskip

\includegraphics[width=.31\textwidth]{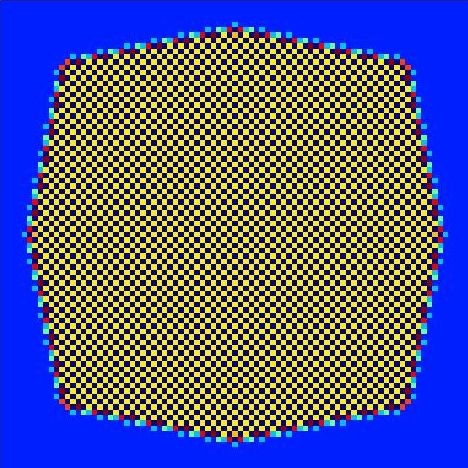} \hspace{0.5mm}
\includegraphics[width=.31\textwidth]{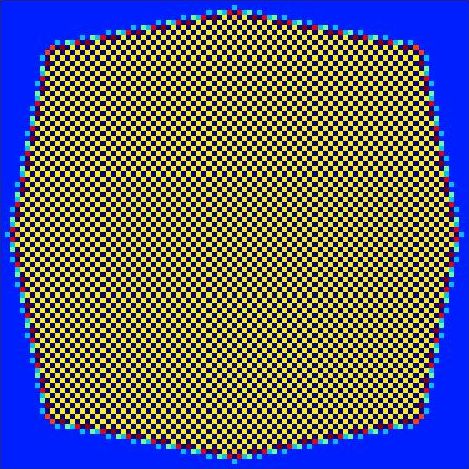} \hspace{0.5mm}
\includegraphics[width=.31\textwidth]{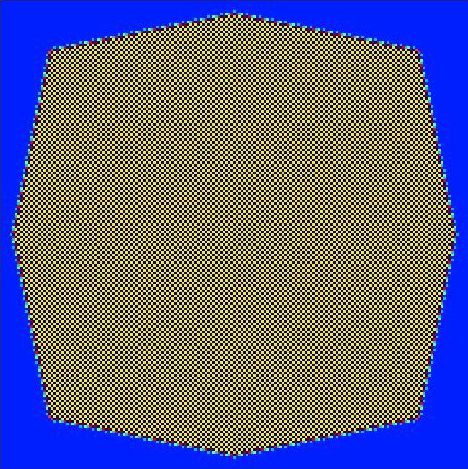}
\caption{The octagon cellular automaton after 5, 15, 25, 55, 65 and 135 time steps, or equivalently, the configurations $\chi^0$, $\chi^1$, $\chi^2$, $\chi^5$, $\chi^6$ and $\chi^{13}$. Black = $e$, yellow = $c$, dark yellow = $c'$, light blue = $p$, dark blue = $\hbar$, light green = $m$, orange = $d'$, red = $d!$, dark red = $q'$.}
\label{octagontimesteps}
\end{figure}

Our induction hypothesis now is: In the octagon cellular automaton, at time $5 + 10i$, $\xi_{5 + 10i} = \chi^i$. In words this says that every 10 time steps, an extra tile is formed. 

The hypothesis can be verified for $i \leq 6$, by performing 65 time steps starting from the initial configuration $\xi_0$. We show the results of this computation in Figure \ref{octagontimesteps}, generated by a computer program of the octagon cellular automaton. 

Suppose now that the hypothesis is true at time $5+10i$, with $i>6$. We will construct $\xi_{5+10(i+1)}$ by performing 10 time steps starting from $\chi^i$. Observe that the label of a cell at time $t+10$ depends only on its own label and that of all cells in $\x + \mathcal{D}_{10}$; we call this set of cells the `10-neighborhood' of $\x$. By the definition of $\chi^i$, we have that for every $i>5$ and for every $\x$, there exists a cell $\y$ such that the labeling of the 10-neighborhood of $\x$ in $\chi^i$ is identical to that of the 10-neighborhood of $\y$ in $\chi^5$. 
Therefore, the label of $\x$ in $\chi^{i+1}$ will be identical to that of $\y$ in $\chi^6$. 
Thus, we can construct $\xi_{5+10(i+1)}$ from $\chi^i$, and we find that indeed, if $\xi_{5 + 10i} = \chi^i$ then $\xi_{5 + 10(i+1)} = \chi^{i+1}$.

Since $\G_t$ is nondecreasing in $t$, we have that for every $t$ there is an $i$ such that $\G_{5+10i} \subseteq \G_t \subseteq \G_{5+10(i+1)}$. 
The radius of $\G_{5+10i}$ is $9 + 6i$. This means that every 10 time steps, the radius increases by 6. We conclude that the limiting shape of $\G_t$ is the octagon $\mathbf{O}$, with scaling function $f(t) = \frac 5{3t}$.

\clearpage
Finally, we prove that $\G_t$ of the octagon cellular automaton is equal to $\T_{t+8} \cup \partial \T_{t+8}$ if $n=3$ and $h \in [5/7, 13/18)$. 

We give a mapping 
$$
\mathcal{M}_o:\{e,\hbar,p,m,d,d',d!,c,c',q,q',s,u\} \to \I,
$$
that maps the state space of the octagon cellular automaton to the mass values of the splitting automaton: (if $1/2 < h <946/1301 \approx 0.727$ then all intervals are nonempty, and moreover, the labels $e$, $\hbar$ and $p$ map to an interval in $[0,1)$, while all other labels map to an interval in $[1,\infty)$):

\begin{eqnarray*}
\mathcal{M}_o(e) & = &0\\
\mathcal{M}_o(\hbar) & = &h\\
\mathcal{M}_o(p) & = &[h + \frac 14 ,1)\\
\mathcal{M}_o(s) & = & [0,1)\\
\mathcal{M}_o(m) & = & [1,4-4h)\\
\mathcal{M}_o(d) & = & [4-4h,16-20h)\\
\mathcal{M}_o(d') & = & [\frac 38 + \frac{5h}4 , 16-20h)\\
\mathcal{M}_o(d!) & = & [\frac 12 + \frac{5h}4 , 16-20h)\\
\mathcal{M}_o(q) & = &[1+h,60-80h)\\
\mathcal{M}_o(q') & = &[\frac{21h}{16} + \frac 78, 60-80h)\\
\mathcal{M}_o(c) & = & [1,60-80h)\\
\mathcal{M}_o(c') & = & [1 + \frac{h}{2}, 60-80h)\\
\mathcal{M}_o(u) & = &[1,60-80h)\\
\end{eqnarray*}

For every $x$, if $n=3$ and $h \in [5/7, 13/18)$ then $\eta_8 \in\mathcal{M}_o^{\xi_0}$. 
This can be verified by tedious, but straightforward inspection: In Figure \ref{initialfigure} we give the configuration at $t=8$ for the splitting automaton with $n=3$ and $h \in [5/7, 13/18)$. 

\begin{figure}[h]
\begin{center}
\scalebox{0.65}{
\begin{tabular}{|c|c|c|c|c|c|c}
111 + 88388$h$ &  &  &  &  & & \\
\hline
0 & 675+128772$h$ & 108+89360$h$ & 81+95692$h$ &  &  & \\
\hline
2610+128408$h$ & 0 & 1350+96824$h$ & 0 & 162+125848$h$ & &  \\
\hline
0 & 4842+116632$h$ & 0 & 1572+112880$h$ & 0 & 81+95692$h$ & \\
\hline
9423+99268$h$ & 0 & 5814+102920$h$ & 0 & 1350+96824$h$   & 108+89360$h$&  \\
\hline
0 & 11700+98608$h$ & 0 & 4842+116632$h$ & 0 & 675+128772$h$ \\
\hline
14592+96512$h$ & 0 & 9423+99268$h$ & 0 & 2610+128408$h$ & 0 & 111 + 88388$h$ \\
\hline
\end{tabular}}
\end{center}
\caption{$\eta_8$ multiplied by $4^8 = 65536$, for the splitting automaton with $n=3$ and $h \in (5/7,13/18)$. Masses not shown are $65536h$.}
\label{initialfigure}
\end{figure}

Next, we will prove by induction in $t$ that $\eta_{t+8}$ is in $\M_o^{\xi_t}$ for all $t$. 
We assume that for some $t$, $\eta_{t+8}$ is in $\M_o^{\xi_t}$. By examining every transition rule, we can then show that as a consequence, $\eta_{t+9}$ is in $\M_o^{\xi_{t+1}}$.

We check one by one the transition rules:

\begin{itemize}
\item{(rule 1)} If in the splitting automaton a site has mass $h$, and none of its neighbors split, then its mass does not change. This is true for 
all $h<1$.
\item{(rule 2)} If an unstable site splits, then by Lemma \ref{checkerboardlemma}, its neighbors do not split. Therefore, it will become empty.
\item{(rule 3)} $h + \frac 14 [1,4-4h) \rightarrow [h+1/4,1)$. This is true for all $h<1$.
\item{(rule 4)} $h + \frac 14 [4-4h,16-20h) \rightarrow [1,4-4h)$. This is true for all $h<1$.
\item{(rule 5)} $h + \frac 14 [\frac 38 + \frac{5h}4,16-20h) \rightarrow [\frac 3{32} + \frac{21h}{16},4-4h)$.
We have that  $[\frac 3{32} + \frac{21h}{16},4-4h) \subseteq [1,4-4h)$ if $1 > h \geq 29/42\approx 06905$.
\item{(rule 6)} $h + \frac 14 [\frac 12 + \frac{5h}4,16-20h) \rightarrow [\frac 18 + \frac{21h}{16},4-4h)$.
We have that  $[\frac 18 + \frac{21h}{16},4-4h) \subseteq [1,4-4h)$ if $1 > h \geq 2/3$.
\item{(rule 7)} $h + \frac 14 [1+h,60-80h) \rightarrow [\frac{5h}4 + \frac 14,15-19h)$.
We have that $[\frac{5h}4 + \frac 14,15-19h) \subseteq [4-4h,16-20h)$ if $0.7143\approx 5/7 \leq h < 1$.
\item{(rule 8)} $h + \frac 14 [1+h,60-80h) + \frac 14[1,4-4h) \rightarrow [\frac{5h}4 + \frac 12,16-20h)$. This is true for all $h<1$.
\item{(rule 9)} $h + \frac 14 [\frac{21h}{16} + \frac 78, 60-80h) + \frac 14[1,4-4h) \rightarrow [\frac{85h}{64} + \frac 7{32},16-20h)$. 
We have that $[\frac{85h}{64} + \frac 7{32},16-20h) \subseteq [\frac{5h}4 + \frac 12,16-20h)$ if $1 > h \geq 5/2$.
\item{(rule 10)} $h + \frac 14 [\frac{5h}4 + \frac 38,16-20h) + \frac 14[1,4-4h) \rightarrow [\frac{21h}{16}+\frac {11}{32},5-5h)$.
We have that $[\frac{21h}{16}+\frac {11}{32},5-5h) \subseteq [\frac{5h}4 + \frac 38,16-20h)$ if $1/2 \leq h \leq 11/15\approx 0.7333$.
\item{(rule 11)} $h + \frac 12 [\frac{5h}4 + \frac 38,16-20h)  \rightarrow [\frac{13h}8 + \frac 3{16},8-9h)$.
We have that $[\frac{13h}8 + \frac 3{16},8-9h) \subseteq [\frac{5h}4 + \frac 38,16-20h)$ if $1/2 \leq h \leq 8/11\approx 0.7273$.
\item{(rule 12)} $[h+ \frac 14,1) + \frac 14 [\frac{5h}4 + \frac 12,16-20h) + \frac 14 [1,4-4h) +  \frac 14 [1,60-80h) \rightarrow [\frac{21h}{16}+\frac 78,21-26h)$.
We have that $[\frac{21h}{16}+\frac 78,21-26h) \subseteq [\frac{21h}{16}+\frac 78,60-80h)$ if $h \leq 13/18\approx 0.7222$.
\item{(rule 13)} $[h+ \frac 14,1) +  \frac 12 [1,4-4h) +  \frac 14 [1,60-80h) \rightarrow [1+h,18-22h)$.
We have that $[1+h,18-22h) \subseteq [1+h,60-80h)$ if $h \leq 21/29\approx 0.7241$.
\item{(rule 14)} $[h+ \frac 14,1) +  \frac 14 [1,4-4h)  \frac 14 [4-4h,16-20h) +  \frac 14 [1,60-80h) \rightarrow [\frac 74,21-26h)$.
We have that $[\frac 74,21-26h) \subseteq [1+h,60-80h)$ if $h \leq 13/18\approx 0.7222$.
\item{(rule 15)} $ \frac 14 [\frac{21h}{16}+\frac 78,60-80h) +  \frac 14 [1,60-80h)  + \frac 14 [\frac{5h}4 + \frac 38,16-20h) \rightarrow [\frac{41h}{64} + \frac 9{16},34-45h)$.
We have that $[\frac{41h}{64} + \frac 9{16},34-45h) \subseteq [1,60-80h)$ if $0.6829\approx 28/41 \leq h \leq 26/35\approx 0.7429$.
\item{(rule 16)} $ \frac 12 [\frac{5h}4 + \frac 12,16-20h)  +  \frac 14 [1+\frac h2,60-80h) \rightarrow [\frac {3h}4 + \frac 12,23-30h)$.
We have that $[\frac {3h}4 + \frac 12,23-30h) \subseteq [1,60-80h)$ if $2/3 \leq h \leq 37/50\approx 0.7400$.
\item{(rule 17)} $ \frac 14 [1+h,60-80h) +  \frac 14 [1,60-80h)  + \frac 14 [\frac{5h}4 + \frac 12,16-20h) \rightarrow [\frac {9h}{16}+ \frac 58,34-45h)$.
We have that $[\frac {9h}{16}+ \frac 58,34-45h) \subseteq [1,60-80h)$ if $2/3 \leq h \leq 26/35\approx 0.7429$.
\item{(rule 18)} $ \frac 14 [\frac{21h}{16}+\frac 78,60-80h) +  \frac 14 [1,60-80h)  + \frac 14 [\frac{5h}4 + \frac 12,16-20h) \rightarrow [\frac {41h}{64}+ \frac {19}{32},34-45h)$.
We have that $[\frac {41h}{64}+ \frac {19}{32},34-45h) \subseteq [1,60-80h)$ if $0.6341\approx 26/41 \leq h \leq 26/35 \approx 0.7429$.
\item{(rule 19)} If in the splitting automaton a cell is empty, and all its neighbors split, then it gets mass at least 1. Since no cell has mass exceeding $60-80h$, that is the maximum mass that an empty cell can get.
\item{(rule 20)} $ \frac 14 [1,4-4h) +  \frac 12 [1+h,60-80h) +  \frac 14 [1,60-80h) \rightarrow [1+ \frac h2,46-61h)$.
We have that $46-61h \leq 60-80h$ if $h \leq 14/19\approx 0.7368$.
\end{itemize}

In summary, all the rules are valid if $h \in [5/7,13/18]$. Therefore, if $\eta_{t+8}$ is in $\M_o^{\xi_t}$ and $5/7 \leq h \leq 13/18$, then $\eta_{t+9}$ is in $\M_o^{\xi_{t+1}}$. This completes the induction.

Finally, by the following observations:
\begin{itemize}
\item all sites in $\G_0$ are in $\T_{8} \cup \partial \T_{8}$,
\item only labels denoted as $u$ map to values in $[1,\infty)$, so for all $t \geq 8$, a site is in $\T_{t}$ if and only if it is in $\T_{8}$ or it has had a label denoted as $u$ at least once before $t$,
\item the label of a site with label $\hbar$ changes into another label if and only if at least one neighbor has a label denoted as $u$,
\item if a site does not have label $\hbar$ at time $t$, then it cannot get label $\hbar$ at any time $t' \geq t$, 
\end{itemize}
we can conclude that if $n=3$ and $h \in [5/7,13/18]$, then $\G_t$ for the octagon cellular automaton is the same set as $\T_{t+8} \cup \partial \T_{t+8}$.

\qed

\section{Explosive and robust regimes}
\label{hvaluessection}

In this section, we prove parts 2 and 3 of Theorem \ref{hvalues}.

\medskip

\noindent{\it Proof of Theorem \ref{hvalues}, part 2}

We will prove that for all $n$, all $h<1/2$ and all $t$, $|\T_{t}| \leq \frac n{1/2-h}$, where by $|\T|$ we denote the cardinality of a set $\T \subset \Z^d$. It follows that
\begin{equation}
|\T| \leq \frac n{1/2-h},
\label{Tbound}
\end{equation}
so that for all $h< 1/2$, the background is robust. 

Let $m_0$ be the total mass in $\T_{t} \cup \partial \T_{t}$ at time 0, and let $m_t$ the total mass in $\T_{t} \cup \partial \T_{t}$ at time $t$.
We have 
$$
m_0 = n + h|\T_{t}| + h|\partial \T_{t}|.
$$ 

At time $t$,  $\T_{t} \cup \partial \T_{t}$ contains a total mass of at least $\frac 1{2d}$ times the number of internal edges in $\T_{t} \cup \partial \T_{t}$. Namely, consider a pair of sites $\x$ and $\y$ connected by an internal edge. Each time that one of them splits, a mass of at least $\frac 1{2d}$ travels to the other one. 

 
The number of internal edges in $\T_{t} \cup \partial \T_{t}$ is at least $d|\T_{t}|$. We demonstrate this by the following argument: Fix an ordering for the $2d$ edges connecting a site to its $2d$ neighbors, such that the first $d$ edges of the origin are in the same closed orthant. For every site $\x$ in $\T_{t}$, all its edges are in $\T_{t} \cup \partial \T_{t}$. If for every $\x \in \T_t$ we count only the first $d$ edges, then we count each edge in $\T_{t} \cup \partial \T_{t}$ at most once, and we arrive at a total of $d|\T_t|$.

Therefore, at least a mass of $d \frac{1}{2d} |\T_{t}| = \frac 12 |\T_{t}|$ remains in $\T_{t} \cup \partial \T_{t}$.
Moreover, since the sites in $\partial \T_{t}$ did not split, the mass $h$ at every site in $\partial \T_{t}$ also remains in $\T_{t} \cup \partial \T_{t}$. Therefore, we have 
$$
m_t \geq \frac 12 |\T_{t}| + h|\partial \T_{t}|.
$$

Finally, we note that since up to time $t$ no mass can have entered or left $\T_{t} \cup \partial \T_{t}$, we have $m_0 = m_t$.
Putting everything together, we find 
$$
n + h|\T_{t}| + h|\partial \T_{t}| \geq \frac 12 |\T_{t}| + h|\partial \T_{t}|,
$$
from which the result follows.
\qed

\medskip

\noindent{\it Proof of Theorem \ref{hvalues}, part 3}

We first give the proof for $d \geq 3$. 
\medskip

First, we need some notation.
Denote by $\mathcal{D}_r \subset \Z^d$ the diamond 
$\mathcal{D}_r = \{x: \sum_i x_i \leq r\}$, and by $\mathcal{L}_r \subset \Z^d$ the layer  $\{x: \sum_i x_i = r\}$.
Denote $d_k = (k,k,\ldots k) \in \Z^d$.
Let $\Gamma_{k,0} = d_k$, and for $i=1\ldots d$, let $\Gamma_{k,i}$ be the set of sites in $\mathcal{L}_{dk+i}$ that have $i$ nearest neighbors in $\Gamma_{k,i-1}$. Observe that for every $i$, $\Gamma_{k,i}$ is not empty, and that $\Gamma_{k,d} = d_{k+1}$.
For example, in dimension 3: $d_1 = (1,1,1)$, $\Gamma_{1,1} = \{(1,1,2), (1,2,1), (2,1,1)\}$, $\Gamma_{1,2} = \{(1,2,2),(2,1,2),(2,2,1)\}$, and $\Gamma_{1,3} = d_2 = (2,2,2)$. 

Let $p_d=\frac{d!}{(2d)^d}\sum_{l=2}^d\frac{(2d)^l)}{l!}$, $q_d=\frac{d!}{(2d)^{d-1}}$, and $h^*$ be defined by $q_d +p_d h^* = 2d(1-h^*)$, so that
$$
h^* = \frac{q_d - 2d}{p_d + 2d}.
$$ 
Finally, we define 
$$
C_d' = \max\{1-\frac 1d,h^*\},
$$
and remark that $C_d' \leq 1 - \frac{3}{4d+2}$, with equality only in the case $d=2$.

We will prove the following statement:

\begin{lemma}
In the splitting automaton on $\Z^d$, if $h \geq C_d'$, and $n \geq 2d(1-h)$, then for every $k \geq 1$, at time $dk+2$, the sites in $\Gamma_{k,1}$ split.
\label{diagonalprogress}
\end{lemma}

Theorem \ref{hvalues}, part 3 follows from Lemma \ref{diagonalprogress} combined with Theorem \ref{rectangular}. Lemma \ref{diagonalprogress} tells us that for every $r$, there is a site on the boundary of the cube $\{x:\max_i x_i\leq r\}$ that splits at least once. By Theorem \ref{rectangular} and by symmetry, all sites in this cube split at least once. Therefore, $\lim_{t \to \infty} \T_t = \Z^d$. 

\begin{proof}

Note that at time $t$, no site outside $\mathcal{D}_t$ can have split yet, so if a site in $\Gamma_{k,i}$ splits at time $dk+i+1$, it does so for the first time. We will show that in fact the sites in $\Gamma_{k,i}$ do split at time $dk+i+1$.
Since we take the parallel splitting order, by symmetry, all sites in $\Gamma_{k,i}$ distribute the same mass when they split; denote by $m(k,i)$ the mass distributed in the first split of a site in $\Gamma_{k,i}$.

We will prove the lemma by induction. For $k=0$, the lemma is true, because we chose $n$ large enough. Now suppose it is true for some value $k$. Then at time $dk+2$, the sites in $\Gamma_{k,2}$ receive  $\frac 2{2d} m(k,1) \geq \frac 1d$, because they each have two neighbors in $\Gamma_{k,1}$. Since they did not split before, their mass is now at least $h + \frac 1d$. Therefore, if $h\geq 1-\frac 1d$, they split at time $dk+3$. This condition is fulfilled because $C_d' \geq 1-\frac 1d$.

Continuing this reasoning, we find for all $i = 2,\ldots,d-1$ that at time $dk+i+1$, the sites in $\Gamma_{k,i}$ split, because each site in $\Gamma_{k,i}$ has $i$ neighbors in $\Gamma_{k,i-1}$, so it receives mass $\frac i{2d} m(k,i-1) \geq \frac i{2d} \geq \frac 1d$.
We calculate, using that for $i=2, ..., d$, we have $m(k,i) = h+\frac{i}{2d} m(k,i-1)$,
$$
m(0,d) = q_d m(0,1)+p_d h \geq q_d +p_d h.
$$  
Recall that $\Gamma_{k,d} = \Gamma_{k+1,0}$. Therefore, at time $d(k+1)+1$ the sites in $\Gamma_{k+1,1}$ receive mass $\frac {m(0,d)}{2d}$. If $h+\frac{m(0, d)}{2d}\geq 1$, then the sites in $\Gamma_{k+1,1}$ split at time $d(k+1)+2$. This condition is fulfilled if $h \geq C_d'$. 
This completes the induction.
Therefore, in Theorem \ref{hvalues}, part 3, for $d \geq 3$ we can take $C_d = C_d'$.
\end{proof}

For $d=2$, $C_d'$ is equal to $1 - \frac{3}{4d+2} = 0.7$. In this case, we can take $C_2 = 13/19 = 0.684\ldots$:

\begin{proposition}
In the splitting automaton with $d=2$, the background is explosive if $h \geq 13/19$.
\end{proposition}

\begin{proof}
In the proof of Theorem \ref{hvalues}, part 3, we have proved that at time $2k+2$, sites $(k+1,k)$ and $(k,k+1)$ split, making use of the fact that at time $2k+1$, site $(k,k)$ splits. We did not take into account that more sites in $\mathcal{L}_k$ might split at time $2k+1$.

We now choose $n \geq 64-84h$, so that at $t=3$, sites $(0,2)$, $(1,1)$ and $(2,0)$ split for the first time. We prove by induction that if $h \geq 13/19$, then at time $2k$, the sites $(k-1,k+1)$ and $(k+1,k-1)$ have mass at least 1, and site $(k,k)$ has mass at least $2-h > 1$. With our choice for $n$, this is true for $k=1$. Now suppose the hypothesis is true for $k$. 
This implies that at time $2k+1$, the sites sites $(k,k+1)$ and $(k+1,k)$ have mass at least $\frac {3h}4 + \frac 34$. If this is at least equal to $4-4h$, then we obtain the induction hypothesis for $k+1$. Solving $\frac {3h}4 + \frac 34 \geq 4-4h$ gives $h \geq 13/19$. 

\end{proof}

{\bf Remark} We have extended this method further, obtaining even smaller bounds for $h$, but as we increase the number of sites we consider, the calculations quickly become very elaborate, and the bound we obtain decreases very slowly. The smallest bound we recorded was $0.683$.

\section{The growth rate for $h < 0$.}
\label{ballsection}

In this section, we prove Theorem \ref{ball}. The proof will follow closely the method used for the proof of Theorem 4.1 in \cite{lionel}, based on the estimates presented in \cite{lionel}, Section 2. An important ingredient to obtain the bounds is that after stabilization, every site has mass at most 1, so that the mass $n$ starting from the origin, must have spread over a minimum number $\lfloor \frac n{1-h} \rfloor$ of sites. On the other hand, as shown in the proof of Theorem 3.1, part 2, we know: $|\T|\leq \frac{n}{\frac{1}{2}-h}$.

But crucial is the use of Green's functions and their asymptotic spherical symmetry, allowing to conclude more about the shape of the set of sites that split. Thus we can derive that $\T$ contains a ball of cardinality comparable to the coarse estimate $\lfloor \frac n{1-h} \rfloor$. Moreover, $\T$ is contained in a ball of cardinality close to $\lfloor \frac n{1/2 - h} \rfloor$. 

The method for $h<0$ does not depend on abelianness or monotonicity, therefore it can be adapted to the splitting model with arbitrary splitting order. 

We start with introducing some notation. 

Denote by $\mathbb{P}_0, \mathbb{E}_0$ as the probability and expectation operator corresponding to the Simple Random Walk $\left\langle X(t)\right\rangle$ starting from the 
origin. For $d\geq 3$, define
$$
g(z)=\mathbb{E}_0\sum_{t=1}^\infty I_{X(t)=z}.
$$
For $d=2$, define 
$$
g_n(z)=\mathbb{E}_0\sum_{t=1}^n I_{X(t)=z},
$$ 
and 
$$
g(z)=\lim_{n \to \infty}[g_n(z)-g_n(0)].
$$
Defining the operator $\Delta$ as
$$
\Delta f(x) = \frac 1{2d} \sum_{y\sim x}f(y) - f(x),
$$

From \cite{lawler}, we have $\Delta g(z)=-1$ when $x$ is the origin and $\Delta g(z)=0$ for all 
other $x\in\mathbb{Z}^d$.

By $u(x)$, we denote the total mass emitted from $x$ during stabilization. 
Then $\Delta u(x)$ is the net increase of mass at site $x$ during stabilization.
For all $x$, let $\eta_\infty(x)$ be the final mass at site $x$ after stabilization.
Since the final mass at each site is strictly less than 1, we have for all $x\in\mathbb{Z}^d$

\begin{equation}
\label{stablemass1}
\Delta u(x)+(n-h)\delta_{0,x} = \eta_\infty(x)- h < 1-h,
\end{equation}
with $\delta_{0,x}=1$ if $x$ is the origin and $0$ for all other $x$.

Moreover, since the final mass at each site $x \in \T$ is in $[0,1)$, we have for all $x \in \T$

\begin{equation}
\label{stablemass2}
-h \leq \Delta u(x)+(n-h)\delta_{0,x} = \eta_\infty(x)- h < 1-h,
\end{equation}
\medskip

\noindent{\it Proof for the inner bound:}

\medskip

For $x\in\mathbb{Z}^d$, $|x|$ is the Euclidean distance from $x$ to the origin.
Let 
$$
\tilde{\xi}_d(x)= (1-h)|x| ^2+(n-h) g(x)  \mbox{ if } d\geq 2,
$$
and let 
$$
\xi_d(x)=\tilde{\xi}_d(x)-\tilde{\xi}_d(\lfloor c_1 r\rfloor e_1),
$$
 with $e_1=(1, 0, 0, ...,0 )$. 
  
From Lemma 2.2 of \cite{lionel}, we have:
$$
\xi_d(x)=O(1), ~ x\in\partial \mathbf{B}_{c_1 r},
$$
therefore there is a constant $C>0$ such that $|\xi_d(x)|<C, ~ x\in\partial \mathbf{B}_{c_1 r}$. Then
$$
 u(x)-\xi_d(x)\geq -\xi_d(x)>-C, ~ x\in\partial \mathbf{B}_{c_1 r}.
$$
Furthermore, using \eqref{stablemass1}, we have for all $x\in\mathbb{Z}^d$,
$$
\Delta(u-\xi_d)=\Delta u-\Delta \xi_d < 1-h-(n-h)\delta_{0,x}-(1-h)-(n-h)\delta_{0,x} = 0.
$$
Therefore, $u-\xi_d$ is superharmonic, which means that it reaches its minimum value on the boundary. Now, as in the proof of Theorem 4.1 of \cite{lionel}, the estimates in Lemmas 2.1 and 2.3 of \cite{lionel} can be applied to conclude that there is a suitable constant $c_2$ such that $u(x)$ is positive for all $x\in \mathbf{B}_{c_1r-c_2}$.
\qed

The proof for the outer bound is more involved that that in \cite{lionel}, because Lemma 4.2 from \cite{lionel}, which is valid for the abelian sandpile growth model with $h \leq 0$, is not applicable for the splitting model. In essence, this lemma uses that if $u(x) > u(y)$ for some $x$ and $y$, then the difference must be at least 1, because mass travels in the form of integer grains. Clearly, we have no such lower bound in the splitting model. 

We note that in \cite{FLP}, a different proof for the outer bound appeared which is valid for the abelian sandpile growth model with $h<d$. Unfortunately, we cannot adapt this proof for the splitting model either. We will comment on this in Section \ref{openproblems}.

We therefore first present some lemma's which we need to prove the outer bound.

\begin{lemma}
Let $h<0$, and take $x_0\in \T$ adjacent to $\partial{\T}$.

There is a path $x_0\sim x_1\sim x_2 \sim \cdots \sim x_m=0$ in $\T$ with
$$
u(x_{k+1}) > u(x_k) - \frac{2d}{2d-1} h, \hspace{2cm} ~ k = 0,\ldots,m-1.
$$
\label{massincrease}
\end{lemma}

\begin{proof}

We will first show that we can find a nearest neighbor path such that:
$$
\frac{2d-1}{2d}u(x_{k+1})-u(x_k)+\frac{1}{2d} u(x_{k-1})\geq -h.
$$
Let $x_1$ be the nearest neighbor of $x_0$ that loses the maximal amount of mass among all the nearest neighbors of $x_0$. If there is a tie, then we make an arbitrary choice. 
Because $x_0$ has at least one neighbor that does not split, 
$\Delta u(x_0)\leq \frac{2d-1}{2d}u(x_1)-u(x_0)$. Therefore:

\begin{equation}
\label{startingpoint}
\frac{2d-1}{2d}u(x_1)-u(x_0)\geq \Delta u(x_0) =\eta_{\infty}(x_0)-h \geq -h.
\end{equation}
For $k \geq 1$ take $x_{k+1} \neq x_{k-1}$ to be the site that looses the maximal amount of mass among all the nearest neighbors of $x_{k}$ 
except $x_{k-1}$. If there is a tie, then we make an arbitrary choice. It is always possible to choose $x_{k+1}$. As long as $x_k$ is not the origin, then we get from \eqref{stablemass2} that:
$$
\frac{2d-1}{2d}u(x_{k+1})+\frac{1}{2d}u(x_{k-1})-u(x_{k})\geq \eta_{\infty}(x_{k})-h \geq -h.
$$

Thus we get a chain $\{x_k,k = 0,1, \ldots\}$ of nearest neighbors, that possibly ends at the origin. 

We rewrite
$$
u(x_{k+1})\geq \frac{2d}{2d-1} u(x_k)- \frac{1}{2d-1} u(x_{k-1}) - \frac{2d}{2d-1} h.
$$
With this expression, and the fact that $u_0 < u_1$ by \eqref{startingpoint}, it is readily derived by induction that $u(x_{k-1}) < u(x_{k})$. Inserting this, we obtain $u(x_{k+1}) > u(x_k) - \frac{2d}{2d-1} h$, so that $u(x_k)$ is strictly increasing in $k$ (recall that $h<0$).

Now it is left to show that the chain does end at the origin. We derive this by contradiction: suppose the 
chain does not visit the origin. Then the chain cannot end, because 
there is always a new nearest neighbor that loses the maximal amount of mass among all the new nearest neighbors. 
But the chain cannot revisit a site that is already in the chain, because $u(x_k)$ is strictly increasing in $k$. But by \eqref{Tbound}, 
the chain cannot visit more than $\frac{n}{1/2 -h}$ sites. 
Therefore, the chain must visit the origin. 
\end{proof}

Define $Q_k(x)=\{y\in\mathbb{Z}^d: \max_i|x_i-y_i|\leq k\}$ as the cube centered at $x$ with radius $k$.
Let
$$
u^{(k)}(x)=(2k+1)^{-d}\sum_{y\in Q_k(x)} u(y)
$$
be the average loss of mass of the sites in cube $Q_k(x)$, and 
$$
\T^{(k)}=\{x: Q_k(x)\subset \T\}.
$$

\begin{lemma}
$\Delta u^{(k)}(x)\geq \frac{k}{2k+1}-h-\frac{(n-h)}{(2k+1)^d}\ind_{0 \in Q_k(x)}$, for all $x\in \T^{(k)}$.
\label{boxaverage}
\end{lemma}

\begin{proof} From Proposition 5.3 of \cite{haiyan}, we know for every $x$:
\begin{equation}
\sum_{y\in Q_k(x)}\eta_\infty(y)\geq \frac{1}{2d}( \textrm{number of internal bounds in } Q_k(x)).
\end{equation}
Equation (\ref{stablemass1}) tells that
$\Delta u(y)=\eta_\infty(y) -h-(n-h)\delta_{0,y}$.
Therefore 
$$
\Delta u^{(k)}(x)=\frac{1}{(2k+1)^d}\sum_{y\in Q_k(x)}[\eta_\infty(y) -h-(n-h)\delta_{0,y}].
$$
Since $Q_k(x)$ has $2dk(2k+1)^{d-1}$ internal bounds, we get:
$$
\Delta u^{(k)}(x)\geq \frac{k}{2k+1}-h-\frac{(n-h)}{(2k+1)^d}\ind_{0 \in Q_k(x)}.
$$
\end{proof}

\begin{lemma}
\label{maxmass}
For every $x \notin \T^{(k)}$,
$$
u(x)< a',
$$
where $a'$ depends only on $k$, $d$ and $h$.
\end{lemma}

\begin{proof}
For $x \notin \T^{(k)}$, there is at least one site $y_0\in Q_k(x)$ that does not split. For $l\geq 1$, take $y_l$ as 
the nearest neighbor of $y_{l-1}$ that loses the maximal amount of mass among all the neighbors of $y_{l-1}$. Since $y_0$ does not split, we have
$$
\frac{1}{2d}\sum_{y \sim y_0}u(y)< 1-h.
$$
Therefore $u(y_1)< 2d(1-h)$. For every $l>1$, we have from \eqref{stablemass1} that $\frac{1}{2d}\sum_{y \sim y_l} u(y)< 1-h+u(y_l)$, 
therefore  $u(y_{l+1})< 2d(1-h)+2d u(y_l)$.
We know there are at most $(2k+1)^d$ sites in $\{y_l\}_{l=0}$. Then:
$$
\max_{x\in Q_k(x)}u(x)< (1-h)\left[ (2d)+(2d)^2+\cdots+(2d)^{(2k+1)^d}\right]< 2(1-h)(2d)^{(2k+1)^d},
$$
so we can choose $a' =  2(1-h)(2d)^{(2k+1)^d}$.
\end{proof}

\noindent{\it Proof of the outer bound:}

\medskip

First, we wish to find an upper bound for $u(x)$ for all $x$ with $c_1' r-1 < |x| \leq c_1' r$, that does not depend on $n$. If $x$ is not in $\T^{(k)}$, then we use Lemma \ref{maxmass}.

For $x \in \T^{(k)}$, take 
$$
\hat{\psi}_d(x) = (\frac{1}{2}-\epsilon-h)|x|^2+(n-h)g(x) \mbox{ if } d\geq 2.
$$
For a fixed small $\epsilon$, we choose $k$ such that
$$
\frac{k}{2k+1} \geq \frac{1}{2}-\epsilon.
$$
For the fixed chosen $k$, define
$$
\tilde{\phi}_d(x)=\frac{1}{(2k+1)^d} \sum_{y\in Q_k(x)} \hat{\psi}_d(y).
$$
Take 
$$
\phi_d(x)=\tilde{\phi}_d(x)-\tilde{\phi}_d(\lfloor c_1'r\rfloor e_1).
$$
By calculation, we obtain $\Delta \phi_d(x)=\Delta \tilde{\phi}_d(x)=1/2-\epsilon-h-(n-h)\ind_{0 \in Q_k(x)}$. Then from Lemma
\ref{boxaverage}, we know
\begin{equation}
\Delta(u^{(k)}-\phi_d)=\Delta u^{(k)}-\Delta \phi_d \geq 0, \forall x \in \T^{(k)}.
\end{equation}
This shows that $u^{(k)}-\phi_d$ is subharmonic on $\T^{(k)}$. So, it takes its maximal value on the boundary.
We combine this information with some lemma's: 
\begin{itemize}
\item Lemma 2.4 of \cite{lionel} gives that for all $x$, $\phi_d(x) \geq -a$ for some constant $a$ depending only on $d$. 
\item Lemma \ref{maxmass} gives that for every $x\in \partial{\T^{(k)}}$, $u(x)< a'$.
\item Finally, from Lemma 2.2 of \cite{lionel}, there is a $\tilde{c}_2$ which only depends on $\epsilon, d$ and $h$, such that for $x$ 
with $c_1'r-1< |x|\leq c_1'r$, $\phi_d(x) \leq \tilde{c}_2$. 
\end{itemize}

The first two lemma's imply that for $x \in \partial{\T^{(k)}}$, $u^{(k)}(x)-\phi_d(x) \leq a' + a$, an upper bound that does not depend on $n$.  
Therefore, since $u^{(k)}-\phi_d$ is subharmonic on $\T^{(k)}$, 
$$
u(x)-\phi_d(x) \leq a' + a, \forall x\in \T^{k}.
$$ 
Combining this with the third lemma, we get: 
$$
u(x) \leq \tilde{c}_2+ a' + a, \forall x\in \mathbf{B}_{c_1'r}\cap \T^{k}.
$$
Therefore, there is an upper bound for $u(x)$ that does not depend on $n$, for all $x\in \mathbf{B}_{c_1'r}\cap \T^{k}$. 
From Lemma 6.3, we know also for $x\notin \T^{k}$, $u(x)<a'$. 
Summarizing all, we obtain that for all $x$ with $c_1'r-1< |x| \leq c_1'r$, 
$u(x)\leq \tilde{C}$, with $\tilde{C}$ a constant that does not depend on $n$.

To summarize, for all $x$ with 
$c_1'r-1< |x| \leq c_1'r$, $u(x)\leq \tilde{C}$, with $\tilde{C}$ a constant that does not depend on $n$.

Now it remains to show that a site that splits, must lie at a bounded distance $c_2'$ from $ \mathbf{B}_{c_1' r}$.
This follows from Lemma \ref{massincrease}: From every site $x_0$ that splits, there is a path along which $u(x)$ increases by an amount of at least $-\frac{2d}{2d-1}h$ every step, and this path continues until the origin. 
Then along the way, this path must cross the boundary of $ \mathbf{B}_{c_1' r}$, and there $u(x) \leq \tilde{C}$. 
Therefore, we can choose $c_2' = - \frac{(2d-1)\tilde{C}}{2dh}$. 
\qed

\section{Open problems}
\label{openproblems}

Based on numerical simulations, we present some tantalizing open problems.

\medskip

\subsection{A critical $h$?}~
\label{criticalh}

\medskip

In Theorem \ref{hvalues}, we give two regimes for $h$ for which we know that the splitting model is explosive resp. robust. In between, there is a large interval for $h$ where we can prove neither. We conjecture however that the two behaviors are separated by a single critical value of $h$, and that this value does not even depend on the splitting order. In dimension 2, our simulations indicate that this critical $h$ is 2/3. 

\medskip

\begin{conjecture}~
\begin{enumerate}
\item For the splitting model on $\Z^d$, there exists a $h_c = h_c(d)$ such that
for all $h < h_c$, the model is robust, and for all $h \geq h_c$, the model is explosive.
\item $h_c(2) = 2/3$.
\end{enumerate}
\end{conjecture}

\medskip

\subsection{The robust regime}~

\medskip

We have proved Theorem \ref{ball} for all $h < 0$. We hoped to extend this result to all $h<1/2$, by adapting the proof used in Section 3.1 of \cite{FLP} for the abelian sandpile growth model (ASGM). However, the first step of this proof uses the fact that $u_n$ is nondecreasing in $n$, where $u_n$ is the total number of topplings that each site performs in stabilizing $\eta^h_n$. This follows from abelianness of the topplings. Since the splitting model is not abelian, we were not able to adapt this proof to work for our model. Nevertheless, we conjecture

\medskip

\begin{conjecture}
Theorem \ref{ball} holds for all $h<1/2$. 
\end{conjecture}

\medskip

\subsection{The explosive regime}~

\medskip

\begin{figure}[h]
\centering
\includegraphics[width=.19\textwidth]{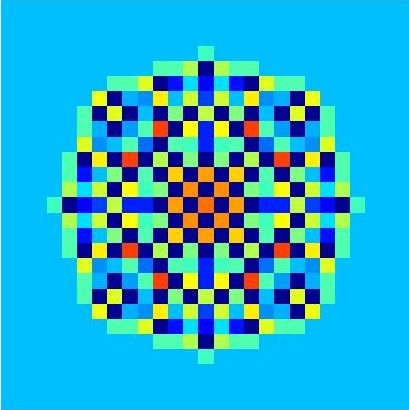}
\includegraphics[width=.19\textwidth]{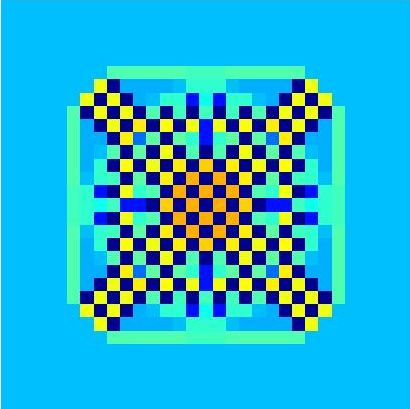}
\includegraphics[width=.19\textwidth]{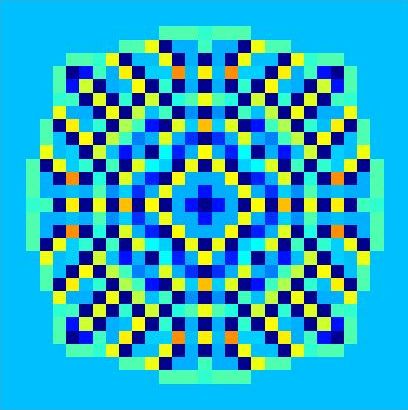}
\includegraphics[width=.19\textwidth]{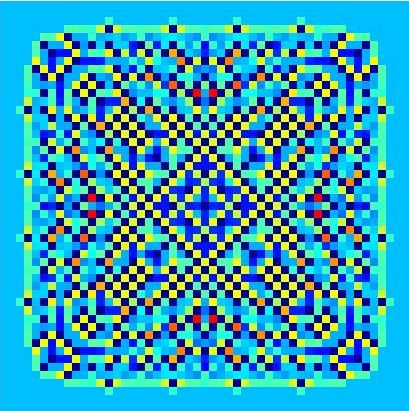}
\includegraphics[width=.19\textwidth]{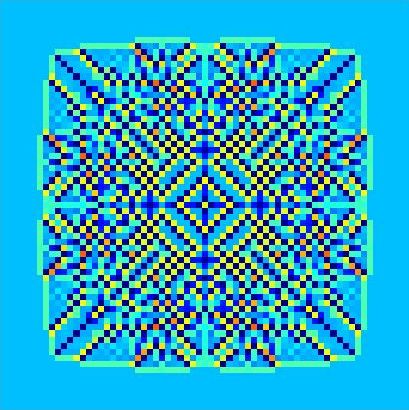}
\caption{The splitting automaton with $h=0.667$ and $n = 16$: From left to right, $t=17$, $t=24$, $t=39$, $t=76$, $t=103$. In this case, a limiting shape may not exist.}
\label{weirdgrowth}
\end{figure}

We have only just started classifying the multitude of shapes of the splitting automaton, that one can observe by varying $h$. We are confident that our method is capable of generating many more limiting shape results. In some cases, we observe that varying $n$ can make a difference, however, we expect the following to be true:

\medskip

\begin{conjecture}
For the splitting automaton on $\Z^d$, for every $h \in [1 - \frac{3}{4d+2}, 1)$ there exists a $n_0$ such that for every $n > n_0$, the limiting shape is a polygon, and depends only on $h$ and $d$.  
\end{conjecture}

This conjecture is reminiscent of Theorem 1 on threshold growth in \cite{gravner}, but the splitting automaton is not equivalent to a two-state cellular automaton. 

For smaller values of $h$, the behavior of the splitting automaton seems to be not nearly as orderly. In Figure \ref{weirdgrowth}, we show the behavior at $h = 0.667$, where we conjecture the model to be explosive. The shape of $\T$ seems to alternate between square and rounded. We are not sure whether a limiting shape exists for this value of $h$. 

\medskip

{\bf Acknowledgement:} We thank Michel Dekking and Ronald Meester for careful reading and helpful suggestions. A.F. thanks the Delft 
University of Technology for hospitality.

\end{document}